\documentclass[12pt]{article}
\usepackage{latexsym, amssymb, amsthm, amsmath}
\usepackage{dsfont}

\DeclareMathAlphabet\gothic{U}{euf}{m}{n}

\makeatletter
\def\eqnarray{\stepcounter{equation}\let\@currentlabel=\theequation
\global\@eqnswtrue
\tabskip\@centering\let\\=\@eqncr
$$\halign to \displaywidth\bgroup\hfil\global\@eqcnt\z@
  $\displaystyle\tabskip\z@{##}$&\global\@eqcnt\@ne
  \hfil$\displaystyle{{}##{}}$\hfil
  &\global\@eqcnt\tw@ $\displaystyle{##}$\hfil
  \tabskip\@centering&\llap{##}\tabskip\z@\cr}

\def\endeqnarray{\@@eqncr\egroup
      \global\advance\c@equation\m@ne$$\global\@ignoretrue}

\def\@yeqncr{\@ifnextchar [{\@xeqncr}{\@xeqncr[5pt]}}
\makeatother

\begin{document}

\newtheorem{lemma}{Lemma}[section]
\newtheorem{thm}[lemma]{Theorem}
\newtheorem{cor}[lemma]{Corollary}
\newtheorem{prop}[lemma]{Proposition}
\newtheorem{ddefinition}[lemma]{Definition}

\theoremstyle{definition}

\newtheorem{remark}[lemma]{Remark}
\newtheorem{exam}[lemma]{Example}

\newcommand{\gota}{\gothic{a}}
\newcommand{\gotb}{\gothic{b}}
\newcommand{\gotc}{\gothic{c}}
\newcommand{\gote}{\gothic{e}}
\newcommand{\gotf}{\gothic{f}}
\newcommand{\gotg}{\gothic{g}}
\newcommand{\gothh}{\gothic{h}}
\newcommand{\gotk}{\gothic{k}}
\newcommand{\gotm}{\gothic{m}}
\newcommand{\gotn}{\gothic{n}}
\newcommand{\gotp}{\gothic{p}}
\newcommand{\gotq}{\gothic{q}}
\newcommand{\gotr}{\gothic{r}}
\newcommand{\gots}{\gothic{s}}
\newcommand{\gott}{\gothic{t}}
\newcommand{\gotu}{\gothic{u}}
\newcommand{\gotv}{\gothic{v}}
\newcommand{\gotw}{\gothic{w}}
\newcommand{\gotz}{\gothic{z}}
\newcommand{\gotA}{\gothic{A}}
\newcommand{\gotB}{\gothic{B}}
\newcommand{\gotG}{\gothic{G}}
\newcommand{\gotL}{\gothic{L}}
\newcommand{\gotS}{\gothic{S}}
\newcommand{\gotT}{\gothic{T}}

\newcounter{teller}
\renewcommand{\theteller}{(\alph{teller})}
\newenvironment{tabel}{\begin{list}%
{\rm  (\alph{teller})\hfill}{\usecounter{teller} \leftmargin=1.1cm
\labelwidth=1.1cm \labelsep=0cm \parsep=0cm}
                      }{\end{list}}

\newcounter{tellerr}
\renewcommand{\thetellerr}{(\roman{tellerr})}
\newenvironment{tabeleq}{\begin{list}%
{\rm  (\roman{tellerr})\hfill}{\usecounter{tellerr} \leftmargin=1.1cm
\labelwidth=1.1cm \labelsep=0cm \parsep=0cm}
                         }{\end{list}}

\newcounter{tellerrr}
\renewcommand{\thetellerrr}{(\Roman{tellerrr})}
\newenvironment{tabelR}{\begin{list}%
{\rm  (\Roman{tellerrr})\hfill}{\usecounter{tellerrr} \leftmargin=1.1cm
\labelwidth=1.1cm \labelsep=0cm \parsep=0cm}
                         }{\end{list}}

\newcounter{proofstep}
\newcommand{\nextstep}{\refstepcounter{proofstep}\vertspace \par 
          \noindent{\bf Step \theproofstep} \hspace{5pt}}
\newcommand{\firststep}{\setcounter{proofstep}{0}\nextstep}

\newcommand{\Ni}{\mathds{N}}
\newcommand{\Qi}{\mathds{Q}}
\newcommand{\Ri}{\mathds{R}}
\newcommand{\Ci}{\mathds{C}}
\newcommand{\Ti}{\mathds{T}}
\newcommand{\Zi}{\mathds{Z}}
\newcommand{\Fi}{\mathds{F}}

\renewcommand{\proofname}{{\bf Proof}}

\newcommand{\vertspace}{\vskip10.0pt plus 4.0pt minus 6.0pt}

\newcommand{\simh}{{\stackrel{{\rm cap}}{\sim}}}
\newcommand{\ad}{{\mathop{\rm ad}}}
\newcommand{\Ad}{{\mathop{\rm Ad}}}
\newcommand{\alg}{{\mathop{\rm alg}}}
\newcommand{\clalg}{{\mathop{\overline{\rm alg}}}}
\newcommand{\Aut}{\mathop{\rm Aut}}
\newcommand{\arccot}{\mathop{\rm arccot}}
\newcommand{\capp}{{\mathop{\rm cap}}}
\newcommand{\rcapp}{{\mathop{\rm rcap}}}
\newcommand{\diam}{\mathop{\rm diam}}
\newcommand{\divv}{\mathop{\rm div}}
\newcommand{\dom}{\mathop{\rm dom}}
\newcommand{\codim}{\mathop{\rm codim}}
\newcommand{\RRe}{\mathop{\rm Re}}
\newcommand{\IIm}{\mathop{\rm Im}}
\newcommand{\tr}{{\mathop{\rm tr \,}}}
\newcommand{\Tr}{{\mathop{\rm Tr \,}}}
\newcommand{\Vol}{{\mathop{\rm Vol}}}
\newcommand{\card}{{\mathop{\rm card}}}
\newcommand{\rank}{\mathop{\rm rank}}
\newcommand{\supp}{\mathop{\rm supp}}
\newcommand{\sgn}{\mathop{\rm sgn}}
\newcommand{\essinf}{\mathop{\rm ess\,inf}}
\newcommand{\esssup}{\mathop{\rm ess\,sup}}
\newcommand{\Int}{\mathop{\rm Int}}
\newcommand{\lcm}{\mathop{\rm lcm}}
\newcommand{\loc}{{\rm loc}}
\newcommand{\HS}{{\rm HS}}
\newcommand{\Trn}{{\rm Tr}}
\newcommand{\n}{{\rm N}}
\newcommand{\SOT}{{\rm SOT}}
\newcommand{\WOT}{{\rm WOT}}

\newcommand{\at}{@}

\newcommand{\spann}{\mathop{\rm span}}
\newcommand{\one}{\mathds{1}}

\hyphenation{groups}
\hyphenation{unitary}

\newcommand{\ca}{{\cal A}}
\newcommand{\cb}{{\cal B}}
\newcommand{\cc}{{\cal C}}
\newcommand{\cd}{{\cal D}}
\newcommand{\ce}{{\cal E}}
\newcommand{\cf}{{\cal F}}
\newcommand{\ch}{{\cal H}}
\newcommand{\chs}{{\cal HS}}
\newcommand{\ci}{{\cal I}}
\newcommand{\ck}{{\cal K}}
\newcommand{\cl}{{\cal L}}
\newcommand{\cm}{{\cal M}}
\newcommand{\cn}{{\cal N}}
\newcommand{\co}{{\cal O}}
\newcommand{\cp}{{\cal P}}
\newcommand{\cs}{{\cal S}}
\newcommand{\ct}{{\cal T}}
\newcommand{\cx}{{\cal X}}
\newcommand{\cy}{{\cal Y}}
\newcommand{\cz}{{\cal Z}}

\newlength{\hightcharacter}
\newlength{\widthcharacter}
\newcommand{\covsup}[1]{\settowidth{\widthcharacter}{$#1$}\addtolength{\widthcharacter}{-0.15em}\settoheight{\hightcharacter}{$#1$}\addtolength{\hightcharacter}{0.1ex}#1\raisebox{\hightcharacter}[0pt][0pt]{\makebox[0pt]{\hspace{-\widthcharacter}$\scriptstyle\circ$}}}
\newcommand{\cov}[1]{\settowidth{\widthcharacter}{$#1$}\addtolength{\widthcharacter}{-0.15em}\settoheight{\hightcharacter}{$#1$}\addtolength{\hightcharacter}{0.1ex}#1\raisebox{\hightcharacter}{\makebox[0pt]{\hspace{-\widthcharacter}$\scriptstyle\circ$}}}
\newcommand{\scov}[1]{\settowidth{\widthcharacter}{$#1$}\addtolength{\widthcharacter}{-0.15em}\settoheight{\hightcharacter}{$#1$}\addtolength{\hightcharacter}{0.1ex}#1\raisebox{0.7\hightcharacter}{\makebox[0pt]{\hspace{-\widthcharacter}$\scriptstyle\circ$}}}

\thispagestyle{empty}

\vspace*{1cm}
\begin{center}
{\Large\bf Weak and strong approximation of semigroups \\[2mm]
on Hilbert spaces} \\[4mm]

\large R. Chill$^1$ and A.F.M. ter Elst$^2$

\end{center}

\vspace{4mm}

\begin{center}
{\bf Abstract}
\end{center}

\begin{list}{}{\leftmargin=1.8cm \rightmargin=1.8cm \listparindent=10mm 
   \parsep=0pt}
\item
For a sequence of uniformly bounded, degenerate semigroups on a Hilbert space, 
we compare various types of convergences to a limit semigroup. 
Among others, we show that convergence of the semigroups, or of the 
resolvents of the generators, in the weak operator topology, 
in the strong operator topology or in certain integral norms are equivalent 
under certain natural assumptions which are frequently met in applications.

\end{list}

\vspace{6mm}
\noindent
\today

\vspace{3mm}

\vspace{3mm}

\vspace{6mm}

\noindent
{\bf Home institutions:}    \\[3mm]
\begin{tabular}{@{}cl@{\hspace{10mm}}cl}
1. & Institut f\"ur Analysis & 
2. & Department of Mathematics  \\
& TU Dresden & 
  & University of Auckland  \\
& 01062 Dresden  &
  & Private bag 92019 \\
& Germany  &
   & Auckland 1142 \\ 
&   &
  & New Zealand \\[8mm]
\end{tabular}

\newpage
\setcounter{page}{1}

\section{Introduction} \label{Swsc1}

The subject of approximation of one-parameter semigroups of operators in various operator 
topologies is a fundamental topic in semigroup theory.
The Trotter--Kato theorem for the approximation of a $C_0$-semigroup in the strong operator 
topology is a classical result which can be found in many textbooks. More recently, the 
question of approximation in the weak operator topology has been studied, too; see, 
for example, Kr\'ol \cite{Kro09}, Eisner \& Sereny \cite{EiSe10} and Furuya \cite{Fur10}.
The purpose of this article is to study the relation between convergence of a sequence 
of semigroups in the weak operator topology and the convergence in the strong 
operator topology.
We concentrate on semigroups in Hilbert spaces whose generators are associated 
with $m$-sectorial forms.
We show that convergence in the weak operator topology and in the strong operator topology 
are equivalent in the case where all involved semigroups are selfadjoint, 
while they are not equivalent in the general case, even when all semigroups are analytic 
and contractive on the same sector.
In the case where all involved semigroups are analytic and contractive on the same 
sector we give additional conditions under which equivalence of convergence in the 
weak operator topology and the strong operator topology does hold.
In fact, equivalence between the two types of convergences holds if in addition the 
semigroups generated by the real parts of the associated forms converge in the 
weak operator topology, or if a monotonicity condition holds which is for example 
satisfied in the context of the Galerkin approximation. 

Motivated by applications to numerical analysis (the Galerkin approximation) or the 
stability of parabolic partial differential equations with respect to the 
underlying (unbounded) domain, we consider not only $C_0$-semigroups but 
general degenerate semigroups, that is, semigroups which are merely defined and strongly continuous 
on the open interval $(0,\infty)$,
and bounded on the open interval $(0,1)$. In this more general context, 
it is for example possible to study the approximation of a ($C_0$-) semigroup on an infinite 
dimensional Hilbert space by degenerate semigroups acting on 
finite dimensional subspaces.
Due to the variational character of the applications which we describe in Sections 4, 5 and 6, 
convergence in the weak operator topology is often easy to establish while convergence in the 
strong operator topology is comparatively more involved, especially when compactness arguments
(obtained by compact embeddings of domains of generators) are not at hand. 
In principle, the additional arguments which allow one to pass from convergence in 
the weak operator topology to convergence in the strong operator topology exist in 
a scattered way in the literature.
In the case of the Galerkin method, these additional arguments are sometimes given, 
but sometimes the reader is left with a statement which does not give the full 
(strong) convergence properties, especially if one is only interested in abstract 
existence results for solutions of parabolic partial differential equations.
The purpose of this article is to gather the arguments in an abstract context and to show that the 
equivalence between convergence in the weak and strong operator topology is a 
general principle independent from the concrete application in numerical analysis, 
the study of parabolic equations on varying domains or in homogenization.

\section{Preliminaries on degenerate semigroups} 
  
Let $X$ be a Banach space.
We call a function $S \colon (0,\infty )\to\cl (X)$, $t\mapsto S_t$, a {\bf degenerate semigroup} if
\begin{enumerate}
 \item[(i)] $S$ is strongly continuous on $(0,\infty )$, 
 \item[(ii)] $S_{t+s} = S_t \, S_s$ for every $t$, $s\in (0,\infty )$, and
 \item[(iii)] $\sup_{t\in (0,1)} \| S_t\| <\infty$.
\end{enumerate}
It is an exercise to show, using properties (ii) and (iii) above, that every 
degenerate semigroup is exponentially bounded, that is, there exist constants 
$M\geq 0$ and $\omega\in\Ri$ such that 
\[
 \| S_t\|\leq M\, e^{\omega t} 
\]
for all $t\in (0,\infty)$.

Let $A\subseteq X\times X$ be a graph in $X$.
Then $-A$ is called the {\bf generator} of a degenerate semigroup $S$ if there exists a 
$\omega\in\Ri$ such that $\lambda I +A$ is boundedly invertible for every 
$\lambda\in (\omega ,\infty )$ and if 
\[
 (\lambda I +A)^{-1}x = \int_0^\infty e^{-\lambda t} S_t x\, dt
\]
for all $x \in X$.
In particular, $\omega$ is chosen large enough so that the Laplace integral on the 
right-hand side converges. 
A crude estimate of the Laplace integral then yields that the pseudoresolvent 
$\lambda\mapsto (\lambda I + A)^{-1}$ satisfies the Hille--Yosida condition
\[
 \| (\lambda -\omega )^k \, (\lambda I + A)^{-k} \, \| \leq M
\]
uniformly for all $k\in\Ni$ and $\lambda \in (\omega,\infty)$.
There seems to be no characterisation of degenerate semigroups on Banach spaces known 
in the literature, for example solely in terms of the Hille--Yosida condition, 
the problem being that pseudoresolvents need not have dense range.
However, there are two important situations in which one has a positive result. 

The first is the situation of degenerate semigroups on {\em reflexive} spaces.
If $A$ is a graph on a reflexive Banach space $X$, if there exists an
$\omega \in \Ri$ such that $\lambda I +A$ is 
boundedly invertible for every $\lambda\in (\omega , \infty )$ and if the pseudoresolvent 
$\lambda\mapsto (\lambda I + A)^{-1}$ satisfies the Hille--Yosida condition above, 
then $-A$ is the generator of a degenerate semigroup $S$ for which, in addition, 
the limit
\[
 Px := \lim_{t\downarrow 0} S_t x 
\]
exists for every $x\in X$ and defines a bounded projection $P$.
Moreover, ${\rm range} \, P = \overline{{\rm range}\, (\lambda I + A)^{-1}}$ and 
${\rm ker} \, P = {\rm ker}\, (\lambda I + A)^{-1}$.
In particular, the range and the kernel of $(\lambda I + A)^{-1}$ do not depend on 
$\lambda$, and $P$ is a projection onto the closure of the domain of~$A$.

The second situation where a characterisation of the generator is available is the 
situation of analytic degenerate semigroups, that is, of 
degenerate semigroups which extend analytically to a sector of the form  
\[
 \Sigma_{\theta} := \{ z\in\Ci\setminus \{ 0\} \colon |{\rm arg}\, z| < \theta \} ,
\]
for some $\theta\in (0,\frac{\pi}{2}]$.
If $A$ is a graph on a Banach space $X$, if there exist $\omega\in\Ri$ and 
$\theta\in (0,\frac{\pi}{2}]$ such that $\lambda I+A$ is boundedly invertible 
for every $\lambda\in \omega + \Sigma_{\frac{\pi}{2} + \theta}$, and if
\[
 \sup_{\lambda \in \omega + \Sigma_{\frac{\pi}{2} + \theta}} \| 
(\lambda -\omega) (\lambda I +A)^{-1} \| <\infty ,
\]
then $-A$ generates a degenerate semigroup $S$ which extends analytically to 
a semigroup on the sector $\Sigma_\theta$.
For the results stated above, see Arendt \cite{Ar01} or Baskakov \cite{Bas04}.

Analytic degenerate semigroups on Hilbert spaces are, for example, generated by graphs associated with closed, (quasi-) sectorial forms, 
and our main results concern indeed solely this particular situation, with the exception of Lemma \ref{lwsc212}, where we 
consider general analytic semigroups on Banach spaces.
By a {\bf form} on a Hilbert space $H$ we mean here a sesquilinear mapping $\gota \colon V\times V \to\Ci$, 
where the {\bf form domain} $V$ is a linear subspace of $H$. 
We point out that the form domain $V$ need not be dense in~$H$.
The {\bf real part} $\Re\gota$ of a form $\gota$ is defined by 
$(\Re\gota) (u,v) := \frac12 (\gota (u,v) + \overline{\gota (v,u)})$, and similarly 
one may define the imaginary part which is, however, not used in this article.
A form $\gota$ is called {\bf sectorial} if there are
$\theta\in (0,\frac{\pi}{2})$ and $\gamma \in \Ri$ such that 
\[
 \gota (u) - \gamma \|u\|_H^2 \in \overline{\Sigma_\theta}
\]
for all $u\in V$, where $\gota(u) = \gota(u,u)$.
We call $\gamma$ a {\bf vertex} of $\gota$.
Finally, a sectorial form $\gota$ on $H$ is called 
{\bf closed} if there exists an $\omega\in\Ri$ 
such that $(u,v) \mapsto (\Re\gota) (u,v) + \omega (u,v)_H$
is a complete inner product on $V$. 

For a closed, sectorial form we define the {\bf associated graph}
\[
 A := \{ (u,f) \in H\times H \colon u\in V \text{ and } 
    \gota (u,v) = (f,v)_H  \mbox{ for all } v\in V \} . 
\]
Then this graph is {\bf $m$-sectorial} in the sense that 
there is an $\omega > 0$ such that 
$\lambda I+A$ is invertible and 
$\| \lambda (\lambda I +A)^{-1} \| \leq 1$ for every 
$\lambda\in \omega + \Sigma_{\theta + \frac{\pi}{2}}$. 
If $\gota$ is symmetric in the sense that $\gota = \Re\gota$, then the associated graph 
is self-adjoint. By applying \cite[Theorem~VI.1.27]{Ka80} to the part of an 
$m$-sectorial graph in the closure of its domain one can see that every 
$m$-sectorial graph is associated to a closed, sectorial form.

\section{Semigroup convergence} \label{Swsc2}

The first main result of this note is the following theorem for self-adjoint graphs 
and semigroups.
It asserts that pointwise convergence of the resolvents in the weak operator topology 
and in the strong operator topology are equivalent, and that the same is true for 
pointwise convergence of semigroups in the weak operator topology and pointwise convergence 
of semigroups in the strong operator topology, uniformly for times in compact subsets of 
$(0,\infty )$.

\begin{thm} \label{twsc201}
Let $H$ be a Hilbert space.
For all $n \in \Ni$ let $A_n$ and $A$ be positive self-adjoint graphs on $H$.
Let $S^{(n)}$ and $S$ be the degenerate semigroups generated by $-A_n$ and $-A$.
Then the following are equivalent.
\begin{tabeleq}
\item \label{twsc201-1}
$\lim_{n \to \infty} S^{(n)}_t = S_t$ in $(\cl(H),\WOT)$ for all $t \in (0,\infty)$.
\item \label{twsc201-2}
$\lim_{n \to \infty} \int_0^T (S^{(n)}_t f,g)_H \, dt = \int_0^T (S_t f,g)_H \, dt$
for all $T > 0$ and $f$, $g \in H$.
\item \label{twsc201-3}
$\lim_{n \to \infty} \int_0^T (S^{(n)}_t f,g(t))_H \, dt = \int_0^T (S_t f,g(t))_H \, dt$
for all $T > 0$, $f \in H$ and $g \in L_1([0,T],H)$.
\item \label{twsc201-4}
$\lim_{n \to \infty} \int_0^T \|(S^{(n)}_t - S_t) f\|_H \, dt = 0$
for all $T > 0$ and $f \in H$.
\item \label{twsc201-5}
$\lim_{n \to \infty} (\lambda \, I + A_n)^{-1} = (\lambda \, I + A)^{-1}$ 
in $(\cl(H),\SOT)$ for all $\lambda \in \Ci$ with $\RRe \lambda > 0$.
\item \label{twsc201-5.2}
There exists a $\lambda \in \Ci$ with $\RRe \lambda > 0$ such that 
$\lim_{n \to \infty} (\lambda \, I + A_n)^{-1} = (\lambda \, I + A)^{-1}$ 
in $(\cl(H),\SOT)$. 
\item \label{twsc201-5.4}
$\lim_{n \to \infty} (\lambda \, I + A_n)^{-1} = (\lambda \, I + A)^{-1}$ 
in $(\cl(H),\WOT)$ for all $\lambda \in \Ci$ with $\RRe \lambda > 0$.
\item \label{twsc201-5.5}
There exists a set $D \subseteq \Ci$ with accumulation point in the open right 
half-plane $\{ \mu\in\Ci : \RRe \mu > 0 \}$ such that 
$\lim_{n \to \infty} (\lambda \, I + A_n)^{-1} = (\lambda \, I + A)^{-1}$ 
in $(\cl(H),\WOT)$ for all $\lambda \in D$.
\item \label{twsc201-5.6}
There exists a $\lambda \in \Ci \setminus \Ri$ with $\RRe \lambda > 0$ such that 
$\lim_{n \to \infty} (\lambda \, I + A_n)^{-1} = (\lambda \, I + A)^{-1}$ 
in $(\cl(H),\WOT)$. 
\item \label{twsc201-6}
For all $f \in H$ and $\delta$, $T > 0$ with $\delta < T$ it follows that 
\[
\lim_{n \to \infty} \sup_{t \in [\delta,T]} \|(S^{(n)}_t - S_t) f\|_H = 0
.  \]
\end{tabeleq}
\end{thm}

Theorem~\ref{twsc201} is for self-adjoint graphs on Hilbert spaces. 
Statements \ref{twsc201-4}, \ref{twsc201-5}, \ref{twsc201-5.2} and \ref{twsc201-6}
are, however, also equivalent for general analytic degenerate semigroups on Banach spaces. 
This is the contents of the next lemma.

\begin{lemma} \label{lwsc212}
Let $X$ be a Banach space.
For all $n \in \Ni$ let $A_n$ and $A$ be graphs on $X$, such that $-A_n$ and 
$-A$ generate analytic degenerate semigroups $S^{(n)}$ and $S$, respectively.
Assume there exists a $\theta \in (0,\frac{\pi}{2}]$ such that the degenerate semigroups 
$S^{(n)}$ and $S$ are uniformly bounded on the same sector $\Sigma_\theta$ with a bound independent of $n\in\Ni$.
Then the following assertions are equivalent.
\begin{tabeleq}
\item \label{lwsc212-2}
$\lim_{n \to \infty} \int_0^T \|(S^{(n)}_t - S_t) f\|_X \, dt = 0$
for all $T > 0$ and $f \in X$.
\item \label{lwsc212-3}
$\lim_{n \to \infty} (\lambda \, I + A_n)^{-1} = (\lambda \, I + A)^{-1}$ 
in $(\cl(X),\SOT)$ for all $\lambda \in \Ci$ with $\RRe \lambda > 0$.
\item \label{lwsc212-4}
There exists a $\lambda \in \Ci$ with $\RRe \lambda > 0$ such that 
$\lim_{n \to \infty} (\lambda \, I + A_n)^{-1} = (\lambda \, I + A)^{-1}$ 
in $(\cl(X),\SOT)$. 
\item \label{lwsc212-1}
For all $f \in X$ and $\delta$, $T > 0$ with $\delta < T$ it follows that 
\[
\lim_{n \to \infty} \sup_{t \in [\delta,T]} \|(S^{(n)}_t - S_t) f\|_X = 0
.  \]
\end{tabeleq}
\end{lemma}

\begin{proof}
`\ref{lwsc212-2}$\Rightarrow$\ref{lwsc212-3}'.
This follows by taking Laplace transforms of the semigroups $S^{(n)}$ and~$S$. 

`\ref{lwsc212-3}$\Rightarrow$\ref{lwsc212-4}'.
Trivial. 

`\ref{lwsc212-4}$\Rightarrow$\ref{lwsc212-3}'.
First Condition~\ref{lwsc212-4} implies that 
$\lim_{n \to \infty} (\lambda \, I + A_n)^{-k} = (\lambda \, I + A)^{-k}$ 
in $(\cl(X),\SOT)$ for every $k\in\Ni$.
Since the powers of $(\lambda \, I + A_n)^{-1}$ coincide, up to a scalar factor, 
with the derivatives of the holomorphic function $\mu \mapsto (\mu \, I + A_n)^{-1}$ at 
the point $\lambda$, Statement~\ref{lwsc212-3} follows by an application of 
Vitali's theorem \cite[Theorem 2.1]{ArNi00}.
(Compare also with \cite[Remark~3.8]{Ar01} for a slightly different argument).

`\ref{lwsc212-3}$\Rightarrow$\ref{lwsc212-1}'.
This follows from Arendt \cite[Theorem~5.2]{Ar01}. 
A different proof is as follows.
First it follows again from Vitali's theorem that 
$\lim_{n \to \infty} (\lambda \, I + A_n)^{-1} = (\lambda \, I + A)^{-1}$ 
in $(\cl(X),\SOT)$ for all $\lambda \in \Sigma_{\frac{\pi}{2} + \theta}$.
Secondly, for every $n\in\Ni$, $f\in X$ and $t>0$, one has the integral representation
\[
 S^{(n)}_t f 
= \frac{1}{2\pi i} \int_\Gamma e^{t \lambda} (\lambda \, I +A_n)^{-1}f \, d\lambda ,
\]
where $\Gamma$ is an appropriately chosen curve in $\Sigma_{\frac{\pi}{2} + \theta}$ connecting 
$e^{\pm i\theta'} \infty$ for some $\theta' \in (\frac{\pi}{2}, \frac{\pi}{2} + \theta )$.
Of course, this integral representation also holds when $S^{(n)}$ and 
$A_n$ are replaced by $S$ and $A$, respectively. 
The strong convergence of $(S^{(n)}_t)$, uniformly for $t$ in intervals 
of the form $[\delta , T]$ now follows from the locally uniform convergence 
of $(\lambda \, I+A_n)^{-1}f$ to $(\lambda \, I + A)^{-1} f$ and a rough estimate of the resolvents for large $\lambda$. 

`\ref{lwsc212-1}$\Rightarrow$\ref{lwsc212-2}'.
This follows from Lebesgue's dominated convergence theorem.
\end{proof}

Now we turn to the proof of the main theorem of this section.

\begin{proof}[{\bf Proof of Theorem \ref{twsc201}.}]
As mentioned above, the equivalence of the statements \ref{twsc201-4}, 
\ref{twsc201-5}, \ref{twsc201-5.2} and \ref{twsc201-6} follows from Lemma~\ref{lwsc212}. 

`\ref{twsc201-1}$\Rightarrow$\ref{twsc201-2}'.
This follows from the Lebesgue dominated convergence theorem.

`\ref{twsc201-2}$\Rightarrow$\ref{twsc201-3}'.
Let $T > 0$.
It follows from Statement~\ref{twsc201-2} that 
\[
\lim_{n \to \infty} \int_0^T ((S^{(n)}_t - S_t) f, \one_{[a,b]} g)_H \, dt = 0
\]
for all $f,g \in H$ and $a,b \in \Ri$ with $0 \leq a < b \leq T$.
Since the step functions are dense in $L_1([0,T],H)$ and the $S_n$ are contractive, 
Statement~\ref{twsc201-3} follows by a $3 \varepsilon$-argument.

`\ref{twsc201-3}$\Rightarrow$\ref{twsc201-4}'.
Let $T > 0$ and $f \in H$.
Then symmetry and the semigroup property give
\begin{eqnarray}
\lefteqn{
\int_0^T \|(S^{(n)}_t - S_t) f\|_H^2 \, dt
} \hspace*{20mm} \nonumber  \\*
& = & \int_0^T ((S^{(n)}_{2t} - S_{2t}) f,f)_H \, dt
   - 2 \RRe \int_0^T ((S^{(n)}_t - S_t) f, S_t f)_H \, dt  \nonumber  \\
& = & \frac{1}{2} \int_0^{2T} ((S^{(n)}_t - S_t) f,f)_H \, dt
   - 2 \RRe \int_0^T ((S^{(n)}_t - S_t) f, S_t f)_H \, dt  
\label{etwsc201;1}
\end{eqnarray}
for all $n \in \Ni$. 
Applying the hypothesis with $g(t) = f$ to the first term
of (\ref{etwsc201;1})
and with $g(t) = S_t f$ to the second term, 
one deduces that both terms on the right-hand side of (\ref{etwsc201;1}) tend to $0$ as $n\to\infty$. 
Hence $\lim_{n \to \infty} \int_0^T \|(S^{(n)}_t - S_t) f\|_H^2 \, dt = 0$.
A simple application of the Cauchy--Schwarz inequality, using the 
contractivity of all involved semigroups, yields also convergence in the $L_1$-sense. 

`\ref{twsc201-4}$\Leftrightarrow$\ref{twsc201-5}$\Leftrightarrow$\ref{twsc201-5.2}$\Leftrightarrow$\ref{twsc201-6}'.
This is a special case of Lemma~\ref{lwsc212}.

`\ref{twsc201-6}$\Rightarrow$\ref{twsc201-1}', `\ref{twsc201-5}$\Rightarrow$\ref{twsc201-5.4}$\Rightarrow$\ref{twsc201-5.5}' 
and `\ref{twsc201-5.4}$\Rightarrow$\ref{twsc201-5.6}'.
Trivial.

`\ref{twsc201-5.5}$\Rightarrow$\ref{twsc201-5.4}'.
This follows from Vitali's theorem.

`\ref{twsc201-5.6}$\Rightarrow$\ref{twsc201-5.2}'.
Let $\lambda \in \Ci \setminus \Ri$ with $\RRe \lambda > 0$ be such that 
$\lim_{n \to \infty} (\lambda \, I + A_n)^{-1} = (\lambda \, I + A)^{-1}$ 
in $(\cl(H),\WOT)$. 
Let $f \in H$.
Then 
\begin{eqnarray*}
\lim_{n \to \infty} \|(\lambda \, I + A_n)^{-1} f\|_H^2
& = & \lim_{n \to \infty} ((\overline \lambda \, I + A_n)^{-1}  
           \, (\lambda \, I + A_n)^{-1} f,f)_H  \\
& = & \lim_{n \to \infty} (\lambda - \overline \lambda)^{-1} \, 
   \Big( ((\overline \lambda \, I + A_n)^{-1} f,f)_H
      - ((\lambda \, I + A_n)^{-1} f,f)_H 
   \Big)   \\
& = & (\lambda - \overline \lambda)^{-1} \, 
   \Big( ((\overline \lambda \, I + A)^{-1} f,f)_H
      - ((\lambda \, I + A)^{-1} f,f)_H 
   \Big)   \\
& = & \|(\lambda \, I + A)^{-1} f\|_H^2
.
\end{eqnarray*}
Then Statement~\ref{twsc201-5.2} is valid.
\end{proof}

One might hope that the convergence in Statement~\ref{twsc201-6} of Theorem~\ref{twsc201}
is valid with $[\delta,T]$ replaced by $(0,T]$.
However, a counterexample has been provided by Daners \cite[Example~6.7]{Da05}.
There is the following characterisation of uniform convergence on $(0,T]$ in 
the strong operator topology.

\begin{lemma} \label{pwsc202}
Assume the assumptions and notation as in Lemma~{\rm \ref{lwsc212}}, and assume in addition that
the Banach space $X$ is reflexive. Suppose the four equivalent statements in Lemma~{\rm \ref{lwsc212}} are valid.
For all $n \in \Ni$ let $P_n$ and $P$ be the projections given by
\[
 P_n f := \lim_{t\downarrow 0} S^{(n)}_t f \text{ and } Pf := \lim_{t\downarrow 0} S_t f \quad (f\in X).
\]
Then the following are equivalent.
\begin{tabeleq}
\item \label{pwsc202-1}
$\lim_{n \to \infty} \sup_{t \in (0,T]} \|(S^{(n)}_t - S_t) f\|_X = 0$
for all $T > 0$ and $f \in X$.
\item \label{pwsc202-2}
There exists a $T > 0$ such that  
$\lim_{n \to \infty} \sup_{t \in (0,T]} \|(S^{(n)}_t - S_t) f\|_X = 0$
for all $f \in X$.
\item \label{pwsc202-3}
$\lim_{n \to \infty} P_n = P$ in $(\cl(X),\SOT)$.
\end{tabeleq}
\end{lemma}

\begin{proof}
`\ref{pwsc202-1}$\Rightarrow$\ref{pwsc202-2}'.
Trivial.

`\ref{pwsc202-2}$\Rightarrow$\ref{pwsc202-3}'.
Since $\lim_{t \downarrow 0} S^{(n)}_t f = P_n f$ for all $f \in X$ and 
$n \in \Ni$, with a similar identity for $S$ and $P$, the 
implication \ref{pwsc202-2}$\Rightarrow$\ref{pwsc202-3} follows by a $3\varepsilon$-argument.

`\ref{pwsc202-3}$\Rightarrow$\ref{pwsc202-1}'.
It follows from the strong resolvent convergence and 
\cite[Theorem~4.2 (b)]{Ar01} that 
$\lim_{n \to \infty} \sup_{t \in (0,T]} \|(S^{(n)}_t - S_t) f\|_X = 0$
for all $T > 0$ and $f \in \overline{\dom A}$. Recall that $P$ is a projection onto $\overline{\dom A}$. 
Now let $T > 0$ and $f \in X$.
Then 
\begin{eqnarray*}
\|S^{(n)}_t f - S_t f\|_X
& = & \|S^{(n)}_t \, P_n f - S_t \, P f\|_X  \\
& \leq & \|S^{(n)}_t \, (P_n f - P f)\|_X + \|(S^{(n)}_t - S_t) P f\|_X  \\
& \leq & M\, \|P_n f - P f\|_X + \|(S^{(n)}_t - S_t) P f\|_X
\end{eqnarray*}
for all $t \in (0,T]$ and $n \in \Ni$, where $M = \sup_{m \in \Ni,\; s \in (0,t]} \|S^{(m)}_s\| < \infty$,
from which Statement~\ref{pwsc202-1} follows.
\end{proof}

In Theorem~\ref{twsc201}\ref{twsc201-5.6} it is essential that $\lambda \not\in \Ri$.
In the next example there is convergence in $(\cl(H),\WOT)$ for $\lambda = 1$, but clearly 
not in $(\cl(H),\SOT)$.
The example is part of \cite[Example~2.3]{EiSe10}.

\begin{exam} \label{xwsc203}
Let $H = \ell_2$.
For all $n \in \Ni$ define $U_n \in \cl(H)$ by 
\[
U_n (x_1,x_2,\ldots) 
= (x_{n+1},\ldots,x_{2n}, x_1,\ldots,x_n, x_{2n+1},\ldots)
.  \]
Then $U_n$ is self-adjoint and 
$\lim_{n \to \infty} U_n = 0$ in $(\cl(H),\WOT)$, but $(U_n)$ does not converge to $0$ 
in $(\cl (H),\SOT)$ since the $U_n$ are also unitary.
Let $V_n = (1 - \frac{1}{n}) U_n$ for all $n \in \Ni$.
Then $\|V_n\| < 1$ and the Cayley transform
$A_n = (I + V_n)(I - V_n)^{-1} \in \cl(H)$ is a positive 
self-adjoint operator.
Moreover, 
$(I + A_n)^{-1}
= \frac{1}{2} \, (I - V_n)
= \frac{1}{2} \, (I - (1 - \frac{1}{n}) U_n)$.
So $\lim_{n \to \infty} (I + A_n)^{-1} = \frac{1}{2} \, I = (I + A)^{-1}$
in $(\cl(H),\WOT)$, where $A = I$ is a positive self-adjoint operator.
However, $\lim_{n \to \infty} (I + A_n)^{-1}$
does not exist in $(\cl(H),\SOT)$, that is Statement~\ref{twsc201-5.2}
in Theorem~\ref{twsc201} is not valid.
\end{exam}

We next present an example that symmetry of the generators in Theorem~\ref{twsc201}
cannot be replaced by $m$-sectoriality, even not with a uniform sector.

\begin{exam} \label{xwsc204}
Let $H$ be an infinite dimensional separable Hilbert space and 
let $S$ be the contraction semigroup defined by $S_t = e^{-t} \, I$
for all $t > 0$.
By \cite[Theorem~2.1]{Kro09} there exists a sequence $(S^{(n)})_{n \in \Ni}$
of unitary $C_0$-groups on $H$ such that for all $T > 0$ one has 
$\lim_{n \to \infty} S^{(n)}_t = S_t$ in $(\cl(H),\WOT)$ uniformly 
for all $t \in (0,T]$.
For all $n \in \Ni$ let $-B_n$ be the generator of $S^{(n)}$ and 
set $B = I$.
Then, by taking Laplace transforms of the respective semigroups, 
$\lim_{n \to \infty} (\lambda \, I + B_n)^{-1} = (\lambda \, I + B)^{-1}$
in $(\cl(H),\WOT)$ for all $\lambda \in \Ci$ with $\RRe \lambda > 0$.
Clearly, for each $t > 0$ 
one does not have $\lim_{n \to \infty} S^{(n)}_t = S_t$ in $(\cl(H),\SOT)$,
since the $S^{(n)}$ are isometric while $S$ is not.
Hence for all $\lambda \in \Ci$ with $\RRe \lambda > 0$ one does not have 
$\lim_{n \to \infty} (\lambda \, I + B_n)^{-1} = (\lambda \, I + B)^{-1}$
in $(\cl(H),\SOT)$, see Lemma~\ref{lwsc212}.

The operator $B_n$ is not invertible in general, but $C_n = I + B_n$ is 
$m$-accretive and invertible for all $n \in \Ni$.
Set $C = I + B = 2I$.
Then $\lim_{n \to \infty} (\lambda \, I + C_n)^{-1} = (\lambda \, I + C)^{-1}$
in $(\cl(H),\WOT)$ for all $\lambda \in \Ci$ with $\RRe \lambda > 0$.
Moreover, for all $\lambda \in \Ci$ with $\RRe \lambda > 0$ one does not have 
$\lim_{n \to \infty} (\lambda \, I + C_n)^{-1} = (\lambda \, I + C)^{-1}$
in $(\cl(H),\SOT)$.
Finally, let $A_n = C_n^{1/2}$ for all $n \in \Ni$.
Then $A_n$ is $m$-sectorial with vertex zero and semiangle~$\frac{\pi}{4}$
by \cite[Theorem~V.3.35]{Ka80} and \cite[Theorem~3.8.3]{ABHN01}.
Since 
\[
(\lambda \, I + A_n)^{-1}
= \frac{1}{\pi} \int_0^\infty
   \frac{\sqrt{\mu}}{\lambda^2 + \mu} \, (\mu \, I + C_n)^{-1} \, d\mu
\]
for all $\lambda \in (0,\infty)$ by \cite[(A1)]{Kat61}, it follows that 
$\lim_{n \to \infty} (\lambda \, I + A_n)^{-1} = (\lambda \, I + \sqrt{2} I)^{-1}$
in $(\cl(H),\WOT)$ for all $\lambda \in (0,\infty)$.
Note that $C_n = A_n^2$ and therefore 
$(I + C_n)^{-1} = (i \, I + A_n)^{-1} \, (- i \, I + A_n)^{-1}$
for all $n \in \Ni$.
So if $\lim_{n \to \infty} (i I + A_n)^{-1}$ converges in $(\cl(H),\SOT)$, then also 
$\lim_{n \to \infty} (-i I + A_n)^{-1}$ converges in $(\cl(H),\SOT)$ 
(compare with the argument in the proof of Lemma \ref{lwsc212} using Vitali's theorem and 
the fact that $\pm i$ lie in the same component of analyticity of the resolvent)
and hence
$\lim_{n \to \infty} (I + C_n)^{-1}$ converges in 
$(\cl(H),\SOT)$, which is a contradiction.
Thus Theorem~\ref{twsc201} cannot be extended to $m$-sectorial operators.
\end{exam}

Instead of considering the square roots $A_n = C_n^\frac12$ one could also consider 
the fractional powers $A_n = C_n^\alpha$ for arbitrary $\alpha\in (0,1)$.
The argument which allows one to pass from the convergence of the resolvents of $C_n$ in 
the weak operator topology to the convergence of the resolvents of $A_n$ in the 
weak operator topology (and back) then simply relies on a functional calculus 
representation of the resolvent of $A_n$ in terms of a contour integral over the 
resolvent of $C_n$ and vice versa \cite{Hs06}.
This means that the angle of sectoriality in the above counterexample can be 
chosen arbitrarily small. 

\smallskip
 
A variant of Theorem \ref{twsc201} is true if, in addition, one also requires 
weak resolvent convergence for the 
real parts of the generators, or, more precisely, for the operators associated 
with the real parts of the involved forms. Under this additional assumption one again has that 
weak resolvent convergence 
implies strong resolvent convergence.

\begin{thm} \label{twsc211}
Let $H$ be a Hilbert space.
For all $n \in \Ni$ let $\gota_n$, $\gota$ be 
closed sectorial sesquilinear forms in $H$ with 
vertex zero.
Let $A_n$, $A$, $R_n$ and $R$ be the $m$-sectorial graphs associated with 
$\gota_n$, $\gota$, $\Re \gota_n$ and $\Re \gota$, respectively.
Suppose there exist $\lambda,\lambda' \in \Ci$ with $\RRe \lambda > 0$,
$\RRe \lambda' > 0$ and $\lambda' \not\in \Ri$ such
that 
\[
\lim_{n\to\infty} (\lambda \, I + A_n)^{-1} = (\lambda I + A)^{-1}
\quad \mbox{in } (\cl(H),\WOT)
\]
and 
\[
\lim_{n\to\infty} (\lambda' \, I + R_n)^{-1} = (\lambda' I + R)^{-1}
\quad \mbox{in } (\cl(H),\WOT)
.  \]
Then $\lim_{n\to\infty} (I + A_n)^{-1} = (I + A)^{-1}$ in $(\cl(H),\SOT)$ and, for
all $f\in X$ and $\delta$, $T>0$ with $\delta <T$,
\[
 \lim_{n\to\infty} \sup_{t\in [\delta ,T]} \| (S_t^{(n)} - S_t )f\|_X = 0 .
\]
\end{thm}

\begin{proof}
It suffices to show that 
$\lim_{n\to\infty} (\lambda \, I + A_n)^{-1} = (\lambda \, I + A)^{-1}$ in 
$(\cl(H),\SOT)$.
Let $f \in H$.
Write $u_n := (\lambda \, I + A_n)^{-1} f$ and $u := (\lambda \, I + A)^{-1} f$
for all $n \in \Ni$.
Then $\lim_{n\to\infty} u_n = u$ weakly in $H$ by the assumed convergence of resolvents in the weak operator topology. Therefore we have to prove that $\lim_{n\to\infty} u_n = u$ (strongly) in $H$.

Let $n \in \Ni$.
Then 
$\gota_n(u_n,v) + \lambda \, (u_n,v)_H = (f,v)_H$
for all $v \in V_n$.
Choosing $v = u_n$ and taking the real part gives
\begin{equation}
\RRe \gota_n(u_n) + (\RRe \lambda) \, \|u_n\|_H^2 
= \RRe (f,u_n)
.
\label{etwsc211;1}
\end{equation}
Similarly
\begin{equation}
\RRe \gota(u) + (\RRe \lambda) \, \|u\|_H^2 
= \RRe (f,u)
.
\label{etwsc211;2}
\end{equation}
Since $\lim_{n\to\infty} u_n = u$ weakly in $H$, it follows from (\ref{etwsc211;1})
that 
\[
(\RRe \lambda) \limsup_{n \to \infty} \|u_n\|_H^2
= \limsup_{n \to \infty} \Big( \RRe (f,u_n) - \RRe \gota_n(u_n) \Big)
= \RRe (f,u) - \liminf_{n \to \infty} \RRe \gota_n(u_n)
.  \]
Now $\Re \gota_n$ and $\Re \gota$ are symmetric closed sesquilinear forms.
Moreover, by assumption and by 
Theorem~\ref{twsc201}\ref{twsc201-5.6}$\Rightarrow$\ref{twsc201-5} one has 
$\lim_{n\to\infty} (\lambda' \, I + R_n)^{-1} = (\lambda' I + R)^{-1}$
in $(\cl(H),\SOT)$.
Using again that $\lim_{n\to\infty} u_n = u$ weakly in $H$, one deduces from 
Attouch \cite[Theorem~3.26]{Att84} 
(or for a shorter proof for forms, see Mosco \cite[Theorem~2.4.1]{Ms94}) the bound
$(\Re \gota)(u) \leq \liminf_{n \to \infty} (\Re \gota_n)(u_n)$.
Therefore 
\[
(\RRe \lambda) \limsup_{n \to \infty} \|u_n\|_H^2
= \RRe (f,u) - \liminf_{n \to \infty} \RRe \gota_n(u_n)
\leq \RRe (f,u) - \RRe \gota(u)
= (\RRe \lambda) \, \|u\|_H^2
,  \]
where we used (\ref{etwsc211;2}) in the last step.
So $\lim_{n\to\infty} u_n = u$ in $H$ and the convergence of resolvents in the strong operator topology follows. The remaining assertion on the convergence of semigroups follows from Lemma \ref{lwsc212}.
\end{proof}

We finish this section by presenting a theorem in which we deal with a
single form $\gota$ which does not have to be symmetric and where the approximation is connected to 
a space approximation of the form domain. It is not a 
corollary to the main theorem (Theorem \ref{twsc201}) nor is it an immediate 
consequence of Theorem \ref{twsc211}. We rather give a variant
of the proof of the latter which does not use the Mosco convergence hidden in the
references to Attouch \cite{Att84} or Mosco \cite{Ms94}.

\begin{thm} \label{twsc301}
Let $V$ and $H$ be Hilbert spaces such that $V$ is continuously 
embedded in~$H$.
Let $\gota \colon V \times V \to \Ci$ be a closed, sectorial, sesquilinear form.
Let $A$ be the $m$-sectorial graph in $H$ associated with $\gota$.
Let $(V_n)$ be an increasing sequence of closed subspaces of $V$ such that
$\bigcup_{n \in \Ni} V_n$ is dense in $V$.
For all $n \in \Ni$ let $\gota_n = \gota|_{V_n \times V_n}$.
Further let $A_n$ be the $m$-sectorial graph in $H$ associated with 
$\gota_n$.
Let $S^{(n)}$ and $S$ be the semigroups generated by $-A_n$ and $-A$, respectively.
Then
\[
\lim_{n \to \infty} \sup_{t \in (0,T]} \|(S^{(n)}_t - S_t) f\|_H = 0
\]
for every $T > 0$ and every $f \in H$.
\end{thm}

\begin{proof}
Without loss of generality we may assume that the form $\gota$ is coercive, that is,
that there exists a $\mu > 0$ such that 
$\mu \, \|u\|_V^2 \leq \RRe \gota(u)$ for all $u \in V$.

Let $f \in H$.
Let $n \in \Ni$.
Set $u_n = (I + A_n)^{-1} f$.
Then $u_n \in V_n$ and 
\[
\gota(u_n,v) + (u_n,v)_H = (f,v)_H
\]
for all $v \in V_n$.
Choose $v = u_n$.
Then $\RRe \gota(u_n) + \|u_n\|_H^2 = \RRe (f,u_n)_H \leq \|f\|_H \, \|u_n\|_H$.
So $\|u_n\|_H \leq \|f\|_H$ and 
$\mu \, \|u_n\|_V^2 \leq \RRe \gota(u_n) \leq \|f\|_H^2$.
Therefore the sequence $(u_n)_{n \in \Ni}$ is bounded in $V$.
Passing to a subsequence if necessary, there exists a $u \in V$
such that $\lim_{n\to\infty} u_n = u$ weakly in $V$.
Let $m \in \Ni$ and $v \in V_m$.
Then 
\[
\gota(u_n,v) + (u_n,v)_H = (f,v)_H
\]
for all $n \in \Ni$ with $n \geq m$.
Take the limit $n \to \infty$.
Then 
\begin{equation}
\gota(u,v) + (u,v)_H = (f,v)_H
.
\label{etwsc301;1}
\end{equation}
Since $\bigcup_{n \in \Ni} V_n$ is dense in $V$ one deduces that 
(\ref{etwsc301;1}) is valid for all $v \in V$.
So $u \in D(A)$ and $u = (I + A)^{-1} f$.
Since
\begin{eqnarray*}
\limsup_{n \to \infty} \|u_n\|_H^2
& = & \limsup_{n \to \infty} \Big( \RRe (f,u_n)_H - \RRe \gota(u_n) \Big)  \\
& = & \RRe (f,u)_H - \liminf_{n \to \infty} \RRe \gota(u_n)  \\ 
& \leq & \RRe (f,u)_H - \RRe \gota(u)
= \|u\|_H^2
\end{eqnarray*}
it follows that $\lim_{n\to\infty} u_n = u$ in $H$.
So $\lim_{n\to\infty} (I + A_n)^{-1} = (I + A)^{-1}$ in $(\cl(H),\SOT)$. 
In particular, the four equivalent statements of Lemma \ref{lwsc212} hold. 

For every $n\in\Ni$ the limits 
\[
 P_n f := \lim_{t\downarrow 0} S_t^{(n)} f \text{ and } Pf := \lim_{t\downarrow 0} S_t f
\]
exist and define projections onto the closures (in $H$) of the domains of the graphs $A_n$ and $A$, 
respectively, and thus by \cite[Theorem~VI.2.1~ii)]{Ka80}
onto the closures of $V_n$ and $V$ in $H$.
Since the graphs $A_n$ and $A$ are associated with forms, one can easily see  
from their definition that the projections $P_n$ and $P$ are orthogonal. 
Using again that $(V_n)$ is increasing and that $\bigcup_{n\in\Ni} V_n$ is dense in $V$, 
we find that $\lim_{n\to\infty} P_n = P$ in $(\cl (H), \SOT )$. 
The claim follows from Lemma \ref{lwsc212} and Lemma \ref{pwsc202}.  
\end{proof}

The situation of Theorem \ref{twsc301} has a flavour of the situation of a monotonically 
decreasing sequence of forms, if one can speak of monotonicity in the context of sectorial forms. 
If the conclusion was strong resolvent convergence of the operators $A_n$, 
it may be seen as a generalisation of Simon \cite[Theorems 3.2 and 4.1]{Si78}. 
By Lemma \ref{lwsc212}, the strong resolvent convergence is equivalent to the 
convergence of the semigroups in the strong operator topology, uniformly in times from 
compact subsets of $(0,\infty )$. 
The uniform convergence up to $t=0$ is an additional feature of Theorem \ref{twsc301}. 
The situation of Theorem \ref{twsc301} is somewhat opposite to the situation of a 
monotonically increasing sequence of forms for which strong resolvent convergence of 
associated operators follows from Kato \cite[Theorem VIII.3.13a]{Ka80} and 
Simon \cite[Theorems 3.1 and 4.1]{Si78} in the symmetric case and from 
Batty \& ter Elst \cite{BaEl14} in a somewhat more general case of sectorial forms.

\section{Galerkin approximation} \label{Swsc3}

One popular situation in which an analytic $C_0$- or degenerate semigroup is approximated by degenerate 
semigroups arises in the numerical analysis of parabolic partial differential equations, namely in the Galerkin approximation, that is,  the space discretization via finite element spaces. It is not necessary to state a separate corollary for this situation since Theorem \ref{twsc301} is precisely designed for it. Given two Hilbert spaces 
$V$ and $H$ such that $V$ is continuously embedded in $H$, and given a closed, sectorial,
sesquilinear form $\gota \colon V \times V \to \Ci$, it suffices to chose an increasing sequence $(V_n)$ of
{\em finite dimensional} subspaces of $V$ (finite element spaces) such that
$\bigcup_{n \in \Ni} V_n$ is dense in $V$. Then Theorem \ref{twsc301} asserts that the semigroups generated by 
the graphs $A_n$ associated with the forms $\gota_n := \gota |_{V_n\times V_n}$ converge in the strong operator topology 
to the semigroup generated by the graph $A$ associated with $\gota$, uniformly for times 
in intervals of the form $(0,T]$. 

In some textbooks on linear and nonlinear analysis, the Galerkin approximation is used as a 
method of proof of existence of solutions of abstract elliptic, parabolic or hyperbolic equations.
The approximate solutions living in finite dimensional subspaces of the energy space usually 
fulfill some a priori estimates, that is, they live in bounded subsets of the form domain or 
in a function space with values in the form domain.
Then an argument using weak compactness allows one to find limit points (in a weak topology), 
and these limit points are shown to be solutions of the original equation.
The question of convergence of the approximate solutions in a norm topology is, 
however, not systematically discussed.
In Dr\'abek--Milota \cite[Proposition 7.2.41]{DrMi13} 
(see also Evans \cite[Theorem 2 in Chapter~2]{Eva90}, 
the strong convergence of approximate solutions of a (nonlinear) stationary problem is 
explicitly stated, while this is not done in the case of approximate solutions of a 
(nonlinear) gradient system in reflexive spaces \cite[Theorem 8.2.5]{DrMi13}.
A statement on strong convergence is also missing in 
Evans \cite[Theorem 3 in \S7.1.2]{Eva98} 
in the context of an abstract {\em linear} parabolic problem, even under 
restrictive assumptions on the finite elements.
Yet, the strong convergence of approximate solutions in the Galerkin 
approximation of abstract parabolic equations is known and Theorem \ref{twsc301} 
is not new in this context: 
see the monograph by Dautray--Lions 
\cite[Remark 5 in Section XVIII.3, p.\ 520]{DaLi92V} with 
uniform convergence in time on bounded intervals of $(0,\infty)$ or the survey by 
Fujita--Suzuki \cite[Theorem 7.1]{FuSu91} with 
uniform convergence in time on compact intervals of $(0,\infty)$, 
in order to mention only two references.

Note, however, that the subspaces $V_n$ in Theorem~\ref{twsc301} do 
not have to be finite dimensional.

\section{Elliptic and parabolic problems on varying domains} \label{Swsc4}

In this section we illustrate Theorems~\ref{twsc211} and \ref{twsc301} by considering a 
sequence of diffusion equations on 
varying open sets $\Omega_n$ which converge monotonically from below to an open set $\Omega$.
We provide new proofs for the next results, which have been 
studied also in Simon \cite[Example 1 and Theorem 4.1 in Section 4]{Si78}. 
In the following, given an open set $\Omega\subseteq\Ri^d$, we consider the Sobolev space 
$H^1_0 (\Omega )$ as a subspace of $H^1 (\Ri^d )$ 
by identifying functions in $H^1_0 (\Omega )$ with functions in $H^1 (\Ri^d)$ which are equal to 
$0$ almost everywhere on $\Ri^d \setminus \Omega$, and similarly we consider the space 
$L_2 (\Omega )$ as a subspace of $L_2 (\Ri^d )$.

\begin{thm} \label{thm.varying}
Let $(\Omega_n)_{n \in \Ni}$ be an increasing sequence of open subsets of $\Ri^d$
and let $\Omega := \bigcup_{n\in\Ni} \Omega_n$.
Let $a\in L_\infty (\Omega ; \Ci^{d\times d} )$ be uniformly elliptic in the sense that there exists $\eta >0$ such that
\[
 \sum_{i,j=1}^d a_{ij}(x) \xi_j \, \overline{\xi_i} \geq \eta \, |\xi |^2 
\text{ for every } \xi\in\Ci^d \mbox{ and } x\in\Omega .
\]
For all $n \in \Ni$ let $u_{0,n}\in L_2 (\Omega_n )\subseteq L_2 (\Omega )$ and $u_0 \in L_2 (\Omega )$.
Further, for all $n \in \Ni$ let $u_n \in C([0,\infty) ; L_2 (\Omega_n ))$ and $u \in C([0,\infty) ; L_2 (\Omega ))$ 
be the solutions of the diffusion equations
\[
\begin{array}{r@{}c@{}ll}
  \partial_t u_n - {\rm div}\, (a(x)\nabla u_n) & {} = {} & 0 & \mbox{ in } (0,\infty ) \times \Omega_n , \\[5pt]
  u_n & = & 0  & \mbox{ on } (0,\infty ) \times \partial\Omega_n , \\[5pt]
  u_n (0,\cdot ) & = & u_{0,n} & \mbox{ in } \Omega_n ,
\end{array}
\]
and
\[
\begin{array}{r@{}c@{}ll}
  \partial_t u - {\rm div}\, (a(x)\nabla u ) & {} = {} & 0 & \mbox{ in } (0,\infty ) \times \Omega , \\[5pt]
  u & = & 0  & \mbox{ on } (0,\infty ) \times \partial\Omega , \\[5pt]
  u (0,\cdot ) & = & u_{0} & \mbox{ in } \Omega ,
\end{array}
\]
respectively. If $\lim_{n\to\infty} \| u_{0,n} - u_0 \|_{L_2 (\Ri^d )} = 0$, then
\[
 \lim_{n\to\infty} \sup_{t\in [0 ,T]} \| u_n(t) - u(t)\|_{L_2 (\Ri^d )} = 0.
\]
\end{thm}

\begin{proof}
For all $n \in \Ni$ we consider the sectorial sesquilinear form 
$\gota_n \colon H^1_0 (\Omega_n ) \times H^1_0 (\Omega_n) \to \Ci$ defined by 
$\gota_n (u,v) = \int_{\Omega_n} (a(x) \nabla u) \cdot \overline{\nabla v}$
and we denote by $A_n$ the sectorial graph on $L_2 (\Omega )$ associated to~$\gota_n$.
Similarly, we define the form $\gota$ on $H^1_0 (\Omega )$ and the 
associated operator $A$, by replacing $\Omega_n$ by $\Omega$ in the above definition.
Observe that $\gota_n$ is the restriction of the form $\gota$ to the space $H^1_0 (\Omega_n )$. 
Observe in addition that $u_n(t) = S^{(n)}_t u_{0,n}$ for all $t > 0$, where $S^{(n)}$ is the 
semigroup generated by $-A_n$, and similarly $u(t)=S_t u_0$ for all $t > 0$, 
where $S$ is the semigroup generated by $-A$.

Since the sequence $(\Omega_n )$ is monotonically increasing to $\Omega$, it is easy to see
that $C_c^\infty (\Omega )$ is a subspace of $\bigcup_n H^1_0 (\Omega_n )$
and thus the latter space
is dense in $H^1_0 (\Omega )$. The claim then follows from Theorem \ref{twsc301}.
\end{proof}

As a consequence of Theorem \ref{thm.varying} and Lemma \ref{lwsc212}, we obtain strong convergence of solutions of elliptic problems, that is, convergence of resolvents in the strong operator topology. 

\begin{cor} \label{cor.varying} 
 Let $(\Omega_n)_{n \in \Ni}$ be an increasing sequence of open subsets of $\Ri^d$
and let $\Omega := \bigcup_{n\in\Ni} \Omega_n$.
Fix $\lambda >0$ and $f\in L_2 (\Omega )$. 
For all $n \in \Ni$ let $u_n\in H^1_0 (\Omega_n)$ be the 
weak solution of the Dirichlet problem
\[
\begin{array}{r@{}c@{}ll}
  \lambda u_n - \Delta u_n &{} = {} & f & \mbox{ in } \Omega_n , \\[5pt]
  u_n & = & 0 & \mbox{ on } \partial\Omega_n . 
\end{array}
\]
Further, let $u\in H^1_0 (\Omega )$ be the weak solution of 
\[
\begin{array}{r@{}c@{}ll}
  \lambda u - \Delta u & {} = {} & f & \mbox{ in } \Omega , \\[5pt]
  u & = & 0 & \mbox{ on } \partial\Omega . 
\end{array}
\]
Then $\lim_{n\to\infty} \| u_n - u \|_{L_2 (\Ri^d )} = 0$. 
\end{cor}

There is a characterisation for convergence of resolvents of the Dirichlet 
Laplacian or resolvents of more general operators as in Theorem~\ref{thm.varying}
in the strong operator topology.
One defines that a sequence $(H^1_0 (\Omega_n ))_{n \in \Ni}$
of Sobolev spaces converges in the sense of Mosco \cite{Mosco1969}
if the following two conditions are valid:
\begin{itemize}
 \item[(i)] 
if $u_n \to u$ weakly in $H^1 (\Ri^d )$ and $u_n\in H^1_0 (\Omega_n)$ for all $n \in \Ni$, 
then $u\in H^1_0 (\Omega )$, and
 \item[(ii)] for every $u\in H^1_0 (\Omega )$ there exists a sequence $(u_n)$ with 
$u_n\in H^1_0 (\Omega_n )$ for all $n \in \Ni$ and $u_n\to u$ in $H^1 (\Ri^d )$;
\end{itemize}
see, for example, \cite[Assumption 6.2]{Da05}. 
Then convergence of resolvents in the strong operator topology is valid if and only if 
the sequence $(H^1_0 (\Omega_n ))$ converges to $H^1_0 (\Omega )$ in the sense of Mosco,
see Daners \cite[Theorem~5.3]{Da03} for the 
implication `$\Leftarrow$' and for the implication `$\Rightarrow$' see 
Attouch \cite[Theorem~3.26]{Att84} or Mosco \cite[Theorem~2.4.1]{Ms94}.

Let us sketch a proof of the implication that Mosco convergence of the Sobolev spaces 
implies strong resolvent convergence. 
In fact, it is not difficult to prove the convergence 
$(\lambda I+A_n)^{-1} \to (\lambda I+A)^{-1}$ in the weak operator topology whenever 
$\lambda >0$. 
Similarly, $(\lambda I+R_n)^{-1} \to (\lambda I+R)^{-1}$ in the 
weak operator topology, where $R_n$ and $R$ are the operators associated with the 
real parts $\Re\gota_n$ and $\Re\gota$, respectively, simply because the latter 
sesquilinear forms have a similar structure as the forms $\gota_n$ and $\gota$ and the 
same arguments apply. 
Once, the two convergences in the weak operator topology are shown, 
the convergence in the strong operator topology follows from Theorem \ref{twsc211}.

\smallskip

Of course, by Lemma \ref{lwsc212} again, the convergence of resolvents in the 
strong operator topology implies convergence of the semigroups in the strong operator topology, 
uniformly in time intervals of the form $[\delta ,T]$ with $\delta, T>0$. 
This is weaker than the convergence of the semigroups in the strong operator topology, 
uniformly in time intervals of the form $(0,T]$, 
as stated in Theorem \ref{thm.varying}. 
In the situation of merely Mosco convergence of the spaces $H^1_0 (\Omega_n )$ 
to $H^1_0 (\Omega )$, one cannot expect uniform convergence on $(0,T]$ in general, 
as the example from \cite[Example~6.7]{Da05} shows.
  
The problem of stability of solutions of elliptic equations with respect to the 
domain has been studied extensively in the works by Bucur \cite{Bu99}, 
Bucur \& Butazzo \cite{BuBu02}, Bucur \& Varchon \cite{BuVa00b}, Daners \cite{Da03}, 
Daners, Hauer \& Dancer \cite{DaHaDa15}, Arrieta \& Barbatis \cite{ArBa14}, 
Arendt \& Daners \cite{ArDa07}, \cite{ArDa08}, Biegert \& Daners \cite{BiDa06}, 
Dal Maso \& Toader \cite{DMTo96}, Sa Ngiamsunthorn \cite{SN12}, \cite{SN12a} and
Mugnolo, Nittka \& Post \cite{MuNiPo13}.

\section{Homogenization on (unbounded) open sets} \label{Swsc5}

We next illustrate Theorem \ref{twsc201} and consider the classical problem of 
homogenization of second-order elliptic operators with periodic coefficients.

For all $k$, $l \in \{ 1,\ldots,d \} $ let $c_{kl} \colon \Ri^d \to \Ri$ be 
measurable and bounded.
Suppose that these coefficients are
\begin{itemize}
\item[(i)] symmetric, that is, $c_{kl} = c_{lk}$ for all $k$, $l \in \{ 1,\ldots,d \}$,
\item[(ii)] periodic, that is, $c_{kl}(x+\gamma) = c_{kl}(x)$ for all 
   $k$, $l \in \{ 1,\ldots,d \} $, $x \in \Ri^d$ and $\gamma \in \Zi^d$, and 
\item[(iii)] uniformly elliptic, that is, there exists a $\mu > 0$ such that
   $\sum_{k,l=1}^d c_{kl}(x) \, \xi_k \, \overline{\xi_l} \geq \mu \, |\xi|^2$ for all 
   $x \in \Ri^d$ and $\xi \in \Ci^d$.
\end{itemize}
For all $\varepsilon > 0$ and $k$, $l \in \{ 1,\ldots,d \}$ 
define $c^{(\varepsilon)}_{kl} \colon \Ri^d \to \Ri$ by 
$c^{(\varepsilon)}_{kl}(x) = c_{kl}(\frac{1}{\varepsilon} \, x)$.
Let $\Omega \subseteq \Ri^d$ be open.
We emphasise that we do not assume that $\Omega$ is bounded.
Let $V$ be a closed subspace of $H^1(\Omega)$ which contains
$C_c^\infty(\Omega)$.
For all $\varepsilon > 0$ let $A_\varepsilon$ be the self-adjoint
operator in $L_2(\Omega)$ associated to the form $\gota_\varepsilon \colon V\times V \to \Ci$
defined by
\[
\gota_\varepsilon(u,v) 
= \int_\Omega \sum_{k,l=1}^d c^{(\varepsilon)}_{kl} \, (\partial_k u) \, \overline{\partial_l v} .
\]
We shall prove that there exists a positive self-adjoint operator $\widehat A$ in 
$L_2(\Omega)$ such that 
\[
\lim_{\varepsilon \downarrow 0} (\lambda \, I + A_\varepsilon)^{-1}
= (\lambda \, I + \widehat A)^{-1}
\]
in the strong operator topology for all $\lambda \in \Ci$ with $\RRe \lambda > 0$.
In fact, we determine the operator $\widehat{A}$ and we only need to prove convergence in the 
weak operator topology.

An explicit description of $\widehat A$ is as follows.
Consider the space $H^1_{\rm per}(\Ri^d)$ of all 
functions $u\in H^1_\loc(\Ri^d)$ which satisfy $u(x+\gamma) = u(x)$ 
for a.e.\ $x \in \Ri^d$ and all $\gamma \in \Zi^d$.
For all $j \in \{ 1,\ldots,d \} $ there exists a $\chi_j \in H^1_{\rm per}(\Ri^d)$
such that 
\begin{equation}
\int_{[0,1]^d} \sum_{k,l=1}^d c_{kl} \, (\partial_k \chi_j) \, \overline{\partial_l v}
= - \int_{[0,1]^d} \sum_{k,l=1}^d c_{jl} \, \overline{\partial_l v}
\label{etwsc405;4}
\end{equation}
for all $v \in H^1_{\rm per}(\Ri^d)$.
Then $\chi_j \in L_\infty(\Ri^d)$ by Stampacchia \cite[Teorema~4.1]{St60}
since $c_{jl} \in L_p$ with 
$p > d$.
For all $k$, $l \in \{ 1,\ldots,d \} $ define 
\[
\hat c_{kl}
= \int_{[0,1]^d} c_{kl}
   - \sum_{j=1}^d \int_{[0,1]^d} c_{kj} \, \partial_j \chi_l
.  \]
It follows from Bensoussan, Lions \& Papanicolau \cite[Remark~1.2.6]{BeLiPa78} 
that there exists a $\mu' > 0$
such that 
$\sum_{k,l=1}^d \hat c_{kl} \, \xi_k \, \overline{\xi_l} \geq \mu' \, |\xi|^2$
for all $\xi \in \Ci^d$.
Let $\widehat A$ be the operator in $L_2(\Omega)$ associated to the form
$\gota \colon V\times V \to\Ci$ defined by
\[
\gota(u,v) 
= \int_\Omega \sum_{k,l=1}^d \hat c_{kl} \, (\partial_k u) \, \overline{\partial_l v} .
\]
The alluded theorem is the following.

\begin{thm} \label{twsc405}
Let $\lambda \in \Ci$ with $\RRe \lambda > 0$.
Then
\[
\lim_{\varepsilon \downarrow 0} (\lambda \, I + A_\varepsilon)^{-1}
= (\lambda \, I + \widehat A)^{-1} \text{ in } (\cl(H),\SOT) .
\]
\end{thm}

\begin{proof}
Let $f \in L_2(\Omega)$ and let $\lambda\in\Ci$ with $\RRe \lambda >0$.
Let $(\varepsilon_n)$ be a sequence of positive real numbers such that 
$\lim_{n\to\infty} \varepsilon_n = 0$.
Let $n \in \Ni$.
Set $u_n = (\lambda \, I + A_{\varepsilon_n})^{-1} f$.
Then 
\begin{equation} 
\int_\Omega \sum_{k,l=1}^d c^{(\varepsilon_n)}_{kl} \, (\partial_k u_n) \, \overline{\partial_l v}
   + \lambda \int_\Omega u_n \, \overline v
= \int_\Omega f \, \overline v
\label{etwsc405;1}
\end{equation}
for all $v \in V$.
Choosing $v = u_n$ gives
\[
\mu \int_\Omega |\nabla u_n|^2 
   + (\RRe \lambda) \int_\Omega |u_n|^2
\leq \RRe \int_\Omega f \, \overline{u_n}
\leq \|f\|_{L_2(\Omega)} \, \|u_n\|_{L_2(\Omega)}
.  \]
So $\|u_n\|_{L_2(\Omega)} \leq (\RRe \lambda)^{-1} \, \|f\|_{L_2(\Omega)}$
and $\mu \int_\Omega |\nabla u_n|^2 \leq (\RRe \lambda)^{-1} \, \|f\|_{L_2(\Omega)}^2$.
Hence the sequence $(u_n)_{n \in \Ni}$ is bounded in $V$.
For all $k \in \{ 1,\ldots,d \} $ and $n \in \Ni$ define 
\[
w_{kn} = \sum_{l=1}^d c^{(\varepsilon_n)}_{kl} \, \partial_l u_n
.  \]
Then the sequence $(w_{kn})_{n \in \Ni}$ is bounded in $L_2(\Omega)$
for all $k \in \{ 1,\ldots,d \} $.
Passing to a subsequence, if necessary, there exist $u \in V$ and 
$w_1,\ldots,w_d \in L_2(\Omega)$ such that 
$\lim_{n\to\infty} u_n = u$ weakly in $V$ and $\lim_{n\to\infty} w_{kn} = w_k$ weakly 
in $L_2(\Omega)$ for all $k \in \{ 1,\ldots,d \} $.

By (\ref{etwsc405;1}) one has 
\[
\sum_{k=1}^d \int_\Omega w_{kn} \, \overline{\partial_k v}
   + \lambda \int_\Omega u_n \, \overline v
= \int_\Omega f \, \overline v
\]
for all $v \in V$ and $n \in \Ni$.
Take the limit $n \to \infty$.
Then 
\begin{equation}
\sum_{k=1}^d \int_\Omega w_k \, \overline{\partial_k v}
   + \lambda \int_\Omega u \, \overline v
= \int_\Omega f \, \overline v
\label{etwsc405;2}
\end{equation}
for all $v \in V$.
We next determine the $w_i$, which requires some work.

Let $i \in \{ 1,\ldots,d \} $.
Let $\varphi \in C_c^\infty(\Omega)$.
For all $n \in \Ni$ define $\chi^{(n)}_i \in H^1_\loc(\Ri^d)$
by $\chi^{(n)}_i(x) = \chi_i(\frac{1}{\varepsilon_n} \, x)$.
We denote by $\pi_k \colon \Ri^d \to \Ri$ the $k$-th coordinate function
for all $k \in \{ 1,\ldots,d \} $.
Let $n \in \Ni$.
Then $\varphi(\pi_i - \varepsilon_n \chi^{(n)}_i) \in H^1_0(\Omega) \subseteq V$.
Moreover, (\ref{etwsc405;1}) gives
\begin{eqnarray}
\int_\Omega f \, \overline{\varphi(\pi_i - \varepsilon_n \chi^{(n)}_i)}
& = & \int_\Omega \sum_{l=1}^d w_{ln} \, 
     \overline{\partial_l (\varphi(\pi_i - \varepsilon_n \chi^{(n)}_i))} 
   + \lambda \int_\Omega u_n \, \overline \varphi \, (\pi_i - \varepsilon_n \chi^{(n)}_i) 
   \nonumber  \\
& = & \int_\Omega \sum_{l=1}^d w_{ln} \, 
     \overline{\partial_l \varphi} \, (\pi_i - \varepsilon_n \chi^{(n)}_i)
   + \int_\Omega \sum_{l=1}^d w_{ln} \, 
     \overline{\varphi} \, \partial_l (\pi_i - \varepsilon_n \chi^{(n)}_i)  \nonumber \\*
& & \hspace*{60mm} {}
   + \lambda \int_\Omega u_n \, \overline \varphi \, (\pi_i - \varepsilon_n \chi^{(n)}_i) 
.  \hspace*{10mm}
\label{etwsc405;3}
\end{eqnarray}
The second term on the right hand side of (\ref{etwsc405;3}) can be rewritten as
\begin{eqnarray*}
\lefteqn{
\int_\Omega \sum_{l=1}^d w_{ln} \, 
     \overline{\varphi} \, \partial_l (\pi_i - \varepsilon_n \chi^{(n)}_i)
} \hspace*{10mm} \\*
& = & \sum_{k,l=1}^d \int_\Omega c^{(\varepsilon_n)}_{kl} \, (\partial_k u_n) \, 
   \overline{\varphi}  \, \partial_l (\pi_i - \varepsilon_n \chi^{(n)}_i)  \\
& = & \sum_{k,l=1}^d \int_\Omega c^{(\varepsilon_n)}_{kl} \, 
    (\partial_k (\overline \varphi \, u_n)) \, 
   \partial_l (\pi_i - \varepsilon_n \chi^{(n)}_i)  
   - \sum_{k,l=1}^d \int_\Omega c^{(\varepsilon_n)}_{kl} \, u_n \, 
   \overline{\partial_k \varphi}  \, \partial_l (\pi_i - \varepsilon_n \chi^{(n)}_i) 
.
\end{eqnarray*}
Note that $\overline \varphi \, u_n \in H^1_0(\Omega) \subseteq H^1(\Ri^d)$
by extending the function with zero. 
Define $v_n \in H^1(\Ri^d)$ by 
$v_n(x) = (\overline \varphi \, u_n)(\varepsilon_n \, x)$.
Then $v_n$ has compact support.
The first term can be simplified since
\[
\sum_{k,l=1}^d \int_\Omega c^{(\varepsilon_n)}_{kl} \, (\partial_k (\overline \varphi \, u_n)) \, 
   \partial_l (\pi_i - \varepsilon_n \chi^{(n)}_i) 
= \frac{1}{\varepsilon_n} \sum_{k,l=1}^d \int_{\Ri^d} c_{kl} \, (\partial_k v_n) \, 
   \partial_l (\pi_i - \chi_i) 
= 0
\]
by (\ref{etwsc405;4}).
So (\ref{etwsc405;3}) gives
\begin{eqnarray}
\int_\Omega f \, \overline{\varphi(\pi_i - \varepsilon_n \chi^{(n)}_i)}
& = & \int_\Omega \sum_{l=1}^d w_{ln} \, 
     \overline{\partial_l \varphi} \, (\pi_i - \varepsilon_n \chi^{(n)}_i)
   - \sum_{k,l=1}^d \int_\Omega c^{(\varepsilon_n)}_{kl} \, u_n \, 
      \overline{\partial_k \varphi}  \, \partial_l (\pi_i - \varepsilon_n \chi^{(n)}_i)   
\nonumber  \\*
& & \hspace*{60mm} {}
   + \lambda \int_\Omega u_n \, \overline \varphi \, (\pi_i - \varepsilon_n \chi^{(n)}_i) 
.
\label{etwsc405;5}
\end{eqnarray}
Now take the limit $n \to \infty$.
Since $\chi_i \in L_\infty(\Ri^d)$ it follows that 
\[
\lim_{n \to \infty} \int_\Omega f \, \overline{\varphi(\pi_i - \varepsilon_n \chi^{(n)}_i)}
= \int_\Omega f \, \overline \varphi \, \pi_i
.  \]
Also 
\[
\Big| \int_\Omega w_{ln} (\partial_l \varphi) \, \chi^{(n)}_i \Big|
\leq \|\chi_i\|_\infty \, \|w_{ln}\|_{L_2(\Omega)} \, \|\partial_l \varphi\|_{L_2(\Omega)}
\]
for all $n \in \Ni$ and the sequence $(w_{ln})_{n \in \Ni}$ is bounded 
in $L_2(\Omega)$.
So 
\[
\lim_{n \to \infty} 
\int_\Omega \sum_{l=1}^d w_{ln} \, 
     \overline{\partial_l \varphi} \, (\pi_i - \varepsilon_n \chi^{(n)}_i)
= \int_\Omega \sum_{l=1}^d w_l \, 
     \overline{\partial_l \varphi} \, \pi_i
.  \]
In order to evaluate the limit of the second term on the right hand side of 
(\ref{etwsc405;5}), we need a lemma.

\begin{lemma} \label{lwsc406}
Let $\tau \colon \Ri^d \to \Ri$ be measurable  and periodic,
i.e.\ $\tau(x + \gamma) = \tau(x)$ for all $x \in \Ri^d$ and $\gamma \in \Zi^d$.
Suppose that $\tau|_{[0,1]^d} \in L_2([0,1]^d)$.
Let $\Omega \subseteq \Ri^d$ be open.
Let $v, v_1,v_2,\ldots \in L_2(\Omega)$ and $K \subseteq \Omega$ compact.
Suppose that $\lim_{n\to\infty} v_n = v$ in $L_2(\Omega)$ and $\supp v_n \subseteq K$
for all $n \in \Ni$.
Let $(\varepsilon_n)$ be a sequence of positive real numbers such that 
$\lim_{n\to\infty} \varepsilon_n = 0$.
Then 
\[
\lim_{n \to \infty} \int_\Omega \tau(\tfrac{1}{\varepsilon_n} \, x) \, v_n
= \Big( \int_{[0,1]^d} \tau \Big) \int_\Omega v
.  \]
\end{lemma}

\begin{proof} 
The proof is similar to the proof in the one dimensional case
in \cite[Example~2.4]{Bra02}.
\end{proof}

We continue with the proof of Theorem~\ref{twsc405}.
Let $k \in \{ 1,\ldots,d \} $.
Then $\lim u_n \, \overline{\partial_k \varphi} = u \, \overline{\partial_k \varphi}$
weakly in $V$, hence weakly in $H^1(\Ri^d)$. 
Since $\supp \overline{\partial_k \varphi}$ is a compact subset of $\Omega$,
it follows that 
$\lim u_n \, \overline{\partial_k \varphi} = u \, \overline{\partial_k \varphi}$
in $L_2(\Omega)$.
Apply Lemma~\ref{lwsc406} with $\tau = c_{kl} \, \partial_l (\pi_i - \chi_i)$.
Then 
\begin{eqnarray*}
\lim_{n \to \infty} \sum_{k,l=1}^d \int_\Omega c^{(\varepsilon_n)}_{kl} \, u_n \, 
   \overline{\partial_k \varphi}  \, \partial_l (\pi_i - \varepsilon_n \chi^{(n)}_i) 
& = & \sum_{k,l=1}^d \Big( \int_{[0,1]^d} c_{kl} \, \partial_l (\pi_i - \chi_i) \Big)
   \int_\Omega u \, \overline{\partial_k \varphi}  \\
& = & - \sum_{k=1}^d \hat c_{ki} \int_\Omega (\partial_k u) \, \overline \varphi
,
\end{eqnarray*}
where we used the definition of the homogenized coefficients and integrated by 
parts.

The last term in (\ref{etwsc405;5}) is easy and 
\[
\lim_{n \to \infty} 
\lambda \int_\Omega u_n \, \overline \varphi \, (\pi_i - \varepsilon_n \chi^{(n)}_i) 
= \lambda \int_\Omega u \, \overline \varphi \, \pi_i
.  \]
Combining the limits, it follows from (\ref{etwsc405;5}) that 
\[
\int_\Omega f \, \overline \varphi \, \pi_i
= \int_\Omega \sum_{l=1}^d w_l \, 
     \overline{\partial_l \varphi} \, \pi_i
   + \sum_{k=1}^d \hat c_{ki} \int_\Omega (\partial_k u) \, \overline \varphi
   + \lambda \int_\Omega u \, \overline \varphi \, \pi_i
.  \]
Next, choosing $v = \varphi \, \pi_i$ in (\ref{etwsc405;2}) gives
\[
\int_\Omega f \, \overline \varphi \, \pi_i
= \sum_{k=1}^d \int_\Omega w_k \, \overline{\partial_k (\varphi \, \pi_i)}
   + \lambda \int_\Omega u \, \overline \varphi \, \pi_i
= \sum_{k=1}^d \int_\Omega w_k \, \overline{\partial_k \varphi} \, \pi_i
   + \int_\Omega w_i \, \overline \varphi
   + \lambda \int_\Omega u \, \overline \varphi \, \pi_i
.  \]
Hence
\[
\int_\Omega w_i \, \overline \varphi
= \sum_{k=1}^d \hat c_{ki} \int_\Omega (\partial_k u) \, \overline \varphi
.  \]
This equality is valid for all $\varphi \in C_c^\infty(\Omega)$.
So $w_i = \sum_{k=1}^d \hat c_{ki} \, \partial_k u$.
Then (\ref{etwsc405;2}) gives
\[
\sum_{k,l=1}^d \int_\Omega \hat c_{kl} \, (\partial_k u) \, \overline{\partial_l v}
   + \lambda \int_\Omega u \, \overline v
= \int_\Omega f \, \overline v
\]
for all $v \in V$.
Therefore $u \in \dom \widehat A$ and $(\lambda \, I + \widehat A) u = f$.

We showed that 
$\lim_{n \to \infty} (\lambda \, I + A_{\varepsilon_n})^{-1} = (\lambda \, I + \widehat A)^{-1}$
in the weak operator topology for all $\lambda \in \Ci$ with $\RRe \lambda > 0$.
Then Theorem~\ref{twsc201} gives that the limit is also valid in the 
strong operator topology.
\end{proof}

We emphasise once again that we do not assume that $\Omega$ is bounded.
Strong resolvent convergence with bounded $\Omega$ has been obtained in 
Bensoussan, Lions \& Papanicolau \cite[Theorem~1.5.1]{BeLiPa78} using much 
more work involving additional correctors.
Strong resolvent convergence on $\Ri^d$ has been proved in \cite[Theorem~1.7]{ZhPa05}
and for bounded $\Omega$ with Dirichlet or Neumann boundary conditions in 
\cite[Theorems~2.3 and~2.8]{ZhPa05}.
A strong convergence of a slightly different nature can be found in Allaire \cite{Al92}.

In the case $V = H^1_0(\Omega)$ one has the following consequence of 
Theorems~\ref{twsc405} and \ref{twsc201}.

\begin{cor} \label{cwsc407}
For all $\varepsilon \in (0,1]$ let $u_{0,\varepsilon} \in L_2(\Omega)$ and let 
$u_0 \in L_2(\Omega)$.
Further, for all $\varepsilon \in (0,1]$ let 
$u_\varepsilon \in C(([0,\infty) ; L_2 (\Omega ))$ 
and $u \in C(([0,\infty) ; L_2 (\Omega ))$ 
be the solutions of the diffusion equations
\[
\begin{array}{r@{}c@{}ll}
  \partial_t u_n + A_\varepsilon u_\varepsilon & {} = {} 
     & 0 & \mbox{ in } (0,\infty ) \times \Omega , \\[5pt]
  u_\varepsilon & = & 0  & \mbox{ on } (0,\infty ) \times \partial\Omega , \\[5pt]
  u_n (0,\cdot ) & = & u_{0,\varepsilon} & \mbox{ in } \Omega ,
\end{array}
\]
and
\[
\begin{array}{r@{}c@{}ll}
  \partial_t u + \widehat A u & {} = {} & 0 & \mbox{ in } (0,\infty ) \times \Omega , \\[5pt]
  u & = & 0  & \mbox{ on } (0,\infty ) \times \partial\Omega , \\[5pt]
  u (0,\cdot ) & = & u_{0} & \mbox{ in } \Omega ,
\end{array}
\]
respectively.
Let $\delta$, $T>0$ with $\delta \leq T$.
If $\lim_{\varepsilon \downarrow 0} \| u_{0,\varepsilon} - u_0 \|_{L_2 (\Omega )} = 0$, then
\[
 \lim_{\varepsilon \downarrow 0} \sup_{t\in [\delta ,T]} 
     \| u_n(t) - u(t)\|_{L_2 (\Omega )} = 0 .
\]
\end{cor}

A similar corollary is valid for Neumann boundary conditions.

\subsection*{Acknowledgements}
The second-named author is most grateful for the hospitality extended
to him during a fruitful stay at the TU Dresden.
He wishes to thank the Institut f\"ur Analysis for financial support.
Part of this work is supported by an
NZ-EU IRSES counterpart fund and the Marsden Fund Council from Government funding,
administered by the Royal Society of New Zealand.
Part of this work is supported by the
EU Marie Curie IRSES program, project `AOS', No.~318910.

 \def\ocirc#1{\ifmmode\setbox0=\hbox{$#1$}\dimen0=\ht0 \advance\dimen0
  by1pt\rlap{\hbox to\wd0{\hss\raise\dimen0
  \hbox{\hskip.2em$\scriptscriptstyle\circ$}\hss}}#1\else {\accent"17 #1}\fi}
  \def\cprime{$'$} \def\cprime{$'$} \def\cprime{$'$}

\end{document}